\documentclass[smallextended,numbook,runningheads]{svjour3}     
\smartqed  
\usepackage{graphicx}
\usepackage{amsmath}
\usepackage{epstopdf} 
\usepackage{mathptmx}      
\usepackage{amsfonts,amssymb} 
\usepackage{algorithm}
\usepackage{algorithmic}
\usepackage{amsmath}
\usepackage{multicol}
\usepackage{multirow}

\newcommand{\CC}{{\mathbb{C}}}

\newcommand{\NN}{{\mathbb{N}}}

\newcommand{\TT}{{\mathbb{T}}}
\newcommand{\KK}{{\mathbb{K}}}
\newcommand{\PP}{{\mathbb{P}}}

\newcommand{\cN}{{\mathcal N}}

\newcommand{\cI}{{\mathcal I}}

\newcommand{\cC}{{\mathcal{C}}}
\newcommand{\nN}{{\mathcal{N}}}
\newcommand{\rR}{{\mathcal{R}}}

\newcommand{\cM}{{\mathcal{M}}}
\newcommand{\cW}{{\mathcal{W}}}
\newcommand{\tc}{{\tilde{c}}}

\newcommand{\rank}{\mathop{\mbox{\rm rank}\,}\nolimits}
\newcommand{\Span}{\mathop{\mbox{\rm span}\,}\nolimits}

\newcommand{\lf}{{\tt lf}}
\newcommand{\Syz}{{\tt Syz}}
\newcommand{\lcm}{{\tt lcm}}
\newcommand{\lm}{{\tt lm}}

\newenvironment{tablehere}
{\def\@captype{table}}
{}
 
\journalname{arXiv submission}
\begin{document}

\title{Numerical computation of H--bases}


\author{Masoumeh Javanbakht         \and
        Tomas Sauer 
}


\institute{M. Javanbakht \at
                Department of Mathematics, 
                Hakim Sabzevari University, 
                397 Sabzevar, Iran \\
              \email{masumehjavanbakht@gmail.com}           
           \and
           T.~Sauer \at
              Lehrstuhl f\"ur Mathematik mit Schwerpunkt Digitale
              Bildverarbeitung \& FORWISS,
              University of Passau,
              Innstr. 43,
              94032 Passau, Germany \\
              \email{Tomas.Sauer@uni-passau.de}     
}

\date{Received: date / Accepted: date}

\maketitle

\begin{abstract}
This paper gives a numerically stable method to compute H-basis which
is based on the computing a minimal basis for the module of syzygies
using singular value decomposition. We illustrate the performance of
this method by means of various examples.
\end{abstract}

\section{Introduction}
The concept of Gr\"obner bases plays an important, if not the
fundamental, role in the development of modern computational Algebraic
Geometry and Computer Algebra systems. Indeed, it provides an
important tool to study and solve numerous problems in many different
areas, ranging from optimization, coding theory, cryptography, to
signal and image processing, robotics, statistics and even more,
cf.~\cite{Buchberger85}.

One of the main drawbacks of Gr\"obner bases, however, is the fact
that their structure depends on a monomial ordering or \emph{term
  order} which breaks the symmetry among the variables and
consequently also loses symmetries present in the underlying
problem. Moreover, Gr\"obner bases are \textbf{not} numerically stable
and even small perturbations in the coefficients of the polynomials
generating an ideal may change the result dramatically,
cf. \cite{MoellerSauerSM00II,Stetter05}. Thus, the development of other
methods with more numerical stability which 
also preserve symmetry is needed.

\emph{H-bases}, or \emph{Macaulay bases} as they are sometimes called
nowadays, are even older than Gr\"obner bases and were
introduced by Macaulay in 1914, cf. \cite{Macaulay16}. Macaulay computed an
H-basis only for a specific example by determining \textit{syzygies} among
the maximal degree homogeneous parts of the polynomials, the so-called
\textit{leading forms}, but he did not give a general symbolic algorithm in
this regard. 
A symbolic algorithm to construct H-bases without relying on monomial orderings
was introduced in \cite{Sauer01}. This algorithm is a quite direct
generalization of Buchberger's algorithm and relies on a
\textit{reduction algorithm} which is a generalization of euclidean
division with remainder to the multivariate case. The crucial point in
this reduction 
algorithm consists of \textit{orthogonalizing} the leading forms
instead of attempting the impossible task of cancelling them. This
generalized reduction leads to a characterization of H-bases which is
based on reducing a generating set of the \emph{syzygy module} of
leading forms. Therefore, determining a basis for the module of syzygies of
finitely many homogeneous forms becomes a crucial part of the
construction of the H-basis. Unfortunately, finding a basis for the
syzygies between forms is far
more intricate than finding a basis for the syzygies between terms.
 
According to our knowledge, there are essentially two general ways to
construct a 
basis for the syzygy module of an ideal. The first one, described in
\cite{Buchberger85}, is based on computing a Gr\"obner basis for the
underlying ideal, while the numerically more suitable approach 
is based on a Linear Algebra and addressed, for example in
\cite{HashemiJavanbaht2017S}; it is referred to as
\textit{degree by degree approach} there. The key component in this
approach is to generate a homogeneous matrix of coefficient vectors of
leading forms of polynomials.

A more general type of such matrices were introduced as \textit{Macaulay
  matrices} and analysed in \cite{batselier14:_canon_decom_c_numer_groeb_border_bases,batselier14:_macaul,batselier14:_null_macaul}. In
\cite{batselier14:_null_macaul}, the degree of regularity of a
polynomial system is 
described and a formula for the dimension of the null space of
Macaulay matrices is derived. A recursive orthogonalization scheme for
two subspaces of these matrices, the range, i.e., the row spaces, and
their null spaces, are 
examined in \cite{batselier14:_macaul}. A decomposition for the vector
space of monomials for a given degree $ k $, and numerical Gr\"oebner
and border bases are finally examined in
\cite{batselier14:_canon_decom_c_numer_groeb_border_bases}.

In this article we describe how one can use a submatrix of Macaulay
matrices to efficiently determine a basis for the syzygy module of
leading forms by techniques from Numerical Linear Algebra and
apply this to develop a numerical algorithm for determining an H-basis,
relying again on well-understood techniques from Numerical Linear
Algebra which are available as efficient and stable implementations
for example in Matlab. Our motivation comes also from the fact that in
practical problems the polynomial systems are often only given as
results of preceding computations as in \cite{Sauer17_Prony} and
therefore ``empirically'', i.e., with inaccurate coefficients,
in the sense of \cite{Stetter05}. 
In this case, using a symbolic algorithm to
compute an H-basis is often not meaningful and may even lead to
misleading results. Moreover the numerical methods are faster than the purely
symbolic ones by orders of magnitude and also have much smaller memory
requirements. We will, however, see in the examples later that there
also exists ill-conditioned ideals where small roundoff errors
contaminate the computations so heavily that the eventual result can even
be wrong.

The paper is organized as follows. After introducing the necessary
concepts and terminology, we will describe the algorithms for syzygy
determination and reduction and prove their validity; moreover, we
will discuss a stopping criterion for the H-basis process. Finally, we
apply the method to various examples of ideals that are known to be
notoriously difficult and serve as benchmarks for ideal basis
computations. The examples will show that there are ideals that are
well-conditioned and that there are numerically ill-conditioned
ideals, the latter due to the fact that some of the normalized
reductions can result in very small remainders that are very hard to
distinguish from zero remainders numerically.

\section{Notations and Definitions}
For a field $\KK$, usually of characteristic zero, we consider the
polynomial ring $ \PP=\KK \left[ x_{1}, \ldots, x_{n} \right]$. Its
subset and linear subspace of homogeneous polynomials of degree $k$ is
defined by
\[
\PP_k := \Span_\KK \{  x^\alpha : |\alpha| = k \},
\]
where the \emph{length} of a multiindex $\alpha \in \NN_0^n$ is
defined as $|\alpha| := \alpha_1 + \cdots + \alpha_n$. Moreover,
\[
\PP_{\leq k} := \bigoplus_{i=0}^k \PP_i
\]
is the linear space of all
polynomials of degree $\le k$.
A monomial $ x^\alpha= x_1^{\alpha_1}\cdots x_n^{\alpha_n}  $ has the
\emph{multidegree} $ \alpha=(\alpha_1, \ldots, \alpha_n)\in \NN_{0}^{n} $ and
the \emph{total degree} $ |\alpha |$. When we speak of the \emph{degree} of a
polynomial, we always mean the total degree, i.e.
\[
\deg (f) := \max \left\{ |\alpha| : f_\alpha \neq 0 \right\}, \qquad f(x) =
\sum_{\alpha \in \NN_0^n} f_\alpha \, x^\alpha.
\]
Let $ \TT $ denote the set of all these monomials, as
well as $ \TT_{k}:= \TT \cap \PP_k =  \{x^{\alpha} : |\alpha|=k\} $. Any
polynomial $ f $ can be written as a finite linear combination
\[
f(x)=\sum_{\alpha\in \NN_{0}^{n}} f_{\alpha} x^{\alpha}, \qquad
f_\alpha \in \KK,
\]
where the coefficients $ f_{\alpha} $ are indexed by using the
standard multi-index notation.
In practical implementations, however, we have to map the multiindices
to $\NN$ by equipping $ \TT $ with a fixed monomial ordering,
cf. \cite{deBoor00}. Since our ordering
has to be compatible with the total degree, we conveniently equip $ \TT $ with
the \emph{graded lexicographical} ordering and any polynomial can
be identified with its coefficient vector. Note that our
results do not depend on the choice of this monomial ordering, it only may
affect the computational efficiency of the implementation.

For $f = f^0 + \cdots + f^k$, $k = \deg f$, $f^i \in \PP_i$, with $f^k
\neq 0$, we call $\lf(f) := f^{k}$ the \emph{leading form} of $f$.
For every ideal $ \cI = \left\langle p_1,\dots,p_s \right\rangle$,
generated by $ p_1, \dots, p_s $, the 
\emph{homogeneous ideal}
\[
\lf( \cI ) := \left\{ \lf (p) : p \in \cI \right\}
\]
can contain polynomials $ p $ such that $\lf(p)\notin \langle
\lf(p_1),\dots, \lf(p_s)\rangle $. The absence of this unwanted
situation leads to the definition of an H-basis. To formulate it, we
recall that a \emph{syzygy} of polynomials $(p_1,\dots,p_s)$ is a
tuple $(q_1,\dots,q_s) \in \PP^s$ such that
$$
\sum_{j=1}^s q_j \, p_j = 0.
$$
By $\Syz (p_1,\dots,p_s)$ we denote the set of all such syzygies, the
\emph{syzygy module} of $p_1,\dots,p_s$.

\begin{definition}\label{D:Hbasis}
  The (finite) set $\lbrace p_1, \dots, p_s\rbrace\subset \PP$ is
  called an H-basis for the ideal $ \cI:=\langle p_1, \dots,
  p_s\rangle $, if one of the following equivalent conditions holds: 
  \begin{enumerate}
    \item\label{it:DHbasis2}
    If $ p\in \cI $ then $ \lf(p)\in \langle \lf(p_1), \dots,
    \lf(p_s)\rangle $.
  \item\label{it:DHbasis1}
    $ p\in \cI $ is equivalent to
    \begin{equation}
      \label{eq:Hrep}
      p = \sum_{i=1}^{s} h_i \, p_i, \qquad h_i\in \PP_{\leq \deg (p)-\deg
        (p_i)}, \quad i=1,\dots,s.
    \end{equation}
     \item\label{it:DHbasis3}
    If
    $ (h_1, \dots, h_s)\in \Syz( \lf(p_1), \dots, \lf(p_s))$ then
    there exist $g_1, \dots, g_s\in \PP$ such that
    $$
    \sum_{i=1}^{s} h_i \, p_i = \sum_{i=1}^{s} g_i \, p_i, \qquad \deg
    (g_ip_i)\leqslant \deg (\sum_{i=1}^{s} h_ip_i), \quad i=1,\dots,
    s.
    $$
  \end{enumerate}
\end{definition}

\begin{remark}
  The restriction that the H--basis is finite is no restriction at
  all. By Hilbert's Basissatz there always exists a finite basis for
  any polynomial ideal $\cI$ which can be transformed into an H--basis
  by a variant of Buchberger's algorithm. Moreover,
  condition~(\ref{it:DHbasis3}) means that any \emph{syzygy} of
  \emph{homogeneous} leading forms can be reduced to zero. A proof of
  the equivalence of the above properties can be found, for example,
  in \cite{MoellerSauer00}.
\end{remark}


\noindent
The condition~(\ref{it:DHbasis3}) allows for a direct extension of
Buchberger's algorithm to compute H-bases, cf. \cite{Sauer01}, but
finding a basis for $ \Syz( \lf(p_1), \dots, \lf(p_s)) $ is crucial
in this extension.
Having at hand a Gr\"obner basis for the ideal $ \langle \lf(p_1), \dots,
\lf(p_s)\rangle $, it is indeed possible to compute such a basis, but
this is unsatisfactory since then one could use the Gr\"obner basis
directly. A more direct way is the degree by degree approach presented in
\cite{HashemiJavanbaht2017S}, which determines an H-basis degree by degree
without having to rely on Gr\"obner bases at all.

To apply the degree by degree approach in our context here,
we need to recall the following definition. 
\begin{definition}\label{D:MacaulayMat}
  Given a set $F=\{ f_1, \dots, f_s\}\subseteq \PP$
  and $k \in \NN_0$, we define the subspace 
  of degree $k$ as
  \begin{equation}
    \label{eq:MacMatDef}
    \cM_k(F) := \left\{ \sum_{i=1}^s h_i \, \lf(f_i) : h_i \in \PP_{k -
        \deg(f_i)},\, 1 \le i \le s \right\} \subseteq \PP_k, 
  \end{equation}
  and the space of homogeneous syzygies of degree $k$ for the leading
  forms as
  \begin{eqnarray}
    \nonumber
    \Syz_k(\lf(F))
    & := & \Syz_k \left(\lf(f_1),\ldots,\lf(f_s) \right) \\
    \label{eq:HomLeadSyz}
    & := & \left\{
      (h_1,\ldots,h_s) : \sum_{i=1}^s h_i \lf(f_i)=0, \, h_i\in
           \PP_{k-\deg(f_i)} \right\}.
  \end{eqnarray}
  In all these definitions we use the convention that $\PP_k = \{ 0
  \}$ whenever $k < 0$.
\end{definition}

\noindent
A generating system for the space $ \cM_k(F) $ is given by the
following matrices.\\
\begin{definition}\label{D:CkDef}
  For $f \in \PP$ we define the matrices
  \[
  C_k(f):=\left[ x^{\alpha} \lf (f) (x) : |\alpha| = k-\deg (f)\right]
  \in \PP_k^{1 \times d_{k-\deg(f)}}, \qquad d_k := {k+n \choose n-1},
  \]
  and their concatenation
  \[
  C_k (F) := \left[ C_k (f_1), \dots, C_k (f_s) \right] \in
  \PP_{k}^{1 \times (d_{k-\deg(f_1)} + \cdots + d_{k-\deg(f_s)})}.
  \]
  Identifying the polynomials with their coefficients with respect to
  the homogeneous monomial basis of degree $k$, we can also assume
  that
  \begin{equation}
    \label{eq:C_k(F)Def}
    C_k (F) \in \KK^{d_k \times d_{k-\deg(f_1)} + \cdots +
    d_{k-\deg(f_s)}},
  \end{equation}
  which is exactly the way how this matrix will be represented on the
  computer with the multiindices in rows and columns ordered in the
  graded lexicographical way.
\end{definition}

\noindent
It can be easily seen that
\begin{equation}
  \label{eq:MSpanC}
  \cM_k (F) = \Span \left\{ C_k (f) : f \in F, \, \deg(f) \le k
  \right\} = \Span C_k (F).
\end{equation}
To simplify \eqref{eq:C_k(F)Def}, we introduce the abbreviation
\[
d_{k - \deg(F)} := \sum_{i=1}^s d_{k - \deg(f_i)},
\]
denote by $\rR \left( C_k(F) \right) \subseteq \PP$ the range of $
C_{k}(F) $ and use
\[
\nN(C_k(F)) = \{ v \in \KK^{d_{k-\deg(F)}} : C_k (F) \, v = 0 \}
\subseteq \KK^{d_{k-\deg(F)}}
\]
for the \emph{null space} or \emph{kernel} of $ C_{k}(F)$. 

\noindent
The space $ \cM_k(F) $ now allows us to establish a connection
between $ \nN \left( C_k(F) \right)$ and $ \Syz_{k}(\lf (F) )$.
Indeed, if $ v \in \nN(C_k(F)) $, we see from the affine column space
interpretation that the expression $ C_{k}(F) v = 0 $ is
equivalent to
\begin{equation}
 \label{eq:VSyz}
 0 = \sum_{i=1}^s \sum_{|\alpha| = k - \deg(f_i)} v_{i,\alpha} \,
 x^\alpha \,\lf (f_i) (x) =: \sum_{i=1}^{s} h_i\lf (f_i)=0.
\end{equation}
Here, the vector $ v $ contains the coefficients of polynomials $ h_i
$. This fact along with the following concept helps us to compute
H-basis degree by degree.

\begin{definition}
  A set $\{p_1,\ldots,p_s\} \subset \PP$ of polynomials is called an
  \emph{H-basis up to (degree) $K$}, $K \in \NN_0$, if for every $k \le K$ and
  any $(h_1,\ldots,h_s) \in \Syz_k(\lf(p_1),\ldots,\lf(p_s))$ there
  exist $g_1,\ldots,g_s\in \PP$ such that
  \[
  \sum_{i=1}^s h_i \, p_i = \sum_{i=1}^s g_i \, p_i, \qquad
  g_i p_i \in \PP_{\leq k-1}, \quad i=1,\dots,s.
  \]
\end{definition}

\noindent
Note that $\{p_1,\ldots,p_s \}$ is an H-basis if and only if it is an
H-basis up to $K$ for every positive integer $K$.

\noindent
This argument shows that finding a basis for $ \nN(C_{k}(F)) $ is
crucial in this approach. On the other hand, one of the most robust
and numerically stable ways to find the orthogonal basis for the null
space is the singular
value decomposition (SVD). The following classical theorem recalls how
the basis vectors for $ \rR(C_k(F)) $ and $\nN(C_k(F)) $ can be read
from SVD of matrix $C_{k}(F) $, cf. \cite{GolubvanLoan96}.

\begin{theorem}
 \label{Theo:SVD}
  Let $ A\in \CC^{m\times n} $ with rank $ A=r $ and let $ A=U\Sigma
  V^{H} $ be SVD of $ A $, then the vectors $ u_1, \dots, u_r $ form
  a basis for $ \rR (A)$ and the vectors $ v_{r+1}, \dots, v_n $ form
  a basis for $ \nN  (A)$.
\end{theorem}


\noindent
Indeed, every element of $ \nN(C_{k}) $ is a linear dependence relation
between the columns of $ C_{k} $ and expresses a syzygy of the form
$\sum_{i=1}^{s} h_i \lf(p_i)=0 $. On the other hand each column of $
C_k $ corresponds to a certain monomial multiple $ x^{\alpha}\lf(p_i)
$. Then, $x_j\sum_{i=1}^{s} h_i \lf(p_i)=0$, which means that all
columns corresponding with $ x_j x^{\alpha}\lf(p_i) =
x^{\alpha+\epsilon_j} \lf (p_i)$ in $ C_{k+1} $ will also be linear
dependent. We will refer to $ h = (h_1,\dots,h_s) $ in this case as a
\emph{basic syzygy} and its monomial multiple $ x_j h $ as an
\emph{extended syzygy}.

\noindent
In the next section we show that how SVD helps us to obtain a basis
for the pure syzygies as well as a basis for the extended syzygies. 

\section{Numerical syzygy computation}
\label{Sec:result}
For polynomials $ F=\{f_1, \dots, f_s\}\subseteq \PP\setminus \{0\} $
let $ d_i:=\deg (f_i)\leq k $, $ i=1, \dots, s $. We now present an
approach to obtain an orthonormal basis for $ \nN(C_{k+1}(F)) $
exploiting the structure of $C_k(F) $ and earlier computations of an
orthogonal basis for $ \nN(C_{k}(F)) $. Since $F$ is the same all over
this computation, we will simply use $C_k$ in the description of the
method and suppose that $ C_{k} $ is a $ t\times q $ matrix whereas $
\cC_{k+1} $ is a $ t' \times q' $ matrix. An orthogonal basis for the
null space of $C_k$ will be denoted by $N_k \in \CC^{q\times d}$
matrix, where $ d:=\dim \nN(C_k) $. By Definition~\ref{D:CkDef}, the
matrix $C_k$ can be partitioned as illustrated below. 


\[
\begin{array}{cccc}
  & \quad\:\,q_1 &   & \quad\;\;\:q_s \\
  C_k= &\multicolumn{3}{c}{t \begin{bmatrix} C_k(f_1)| & \cdots
      &|C_k(f_s) \end{bmatrix},}
\end{array}
\]
where $q_{j}:=\#\TT_{k-d_j} = d_{k-\deg(f_j)}$, $ j=1, \dots, s $, and
$t=d_k$.
Consequently $ N_k $ can be partitioned as  
\[
N_k = \begin{bmatrix} 
  B_{1} \\ \hline
  \vdots \\ \hline
  B_{s}
\end{bmatrix}, \qquad B_j \in \KK^{q_j \times d}.
\]
Now let $ L_{ij} \in \KK^{d_{k-d_i} \times d_{k-d_i+1}}$ be the
\emph{shift matrix} that represents the multiplication with $ x_j $, $
j=1,\dots, n $, $ i=1,\dots, s $. Then the block diagonal matrix 
\[
L_j := \begin{bmatrix} L_{1j} & &\\
  & \ddots &  \\
  &  &  L_{sj}\\
  0 &\cdots & 0 \end{bmatrix} \in \KK^{d' \times d}
\] 
has the property that
\begin{equation}
  v\in \nN(C_k) \qquad \Rightarrow \qquad  w := L_j v  \in
  \nN(C_{k+1}), \quad j=1,\dots,n,
\end{equation}
Setting
\[
A:=
\begin{bmatrix}
  L_1 &\cdots& L_n
\end{bmatrix}\, N_k
\]
and $ r:= $rank $ A $, the SVD of the matrix $ A $ will be
\[
A = Q S W^{H}, \qquad Q \in \CC^{d'\times d'}, \, W \in \CC^{d\times d}.
\]
Now $ Q $ can be partitioned as
\begin{equation}
  \label{eq:DeqQ}
Q =
\begin{bmatrix}
  Q_1 \,|\, Q_2
\end{bmatrix}, \qquad Q_1 \in \CC^{d \times r}, \, Q_2 = \CC^{d \times
  d-r},
\end{equation}
and we can compute yet another SVD
\[
B := C_{k+1} Q_2 = U \Sigma V^H.
\]
Let $ r':= \rank B$, partition $ V $ as
\begin{equation}
   \label{eq:DeqV}
V =
\begin{bmatrix}
  V_1 \, |\, V_2
\end{bmatrix}, \qquad V_1 \in \CC^{d' \times r'}, \, V_2 \in
\CC^{d' \times d'-r'},
\end{equation}
and define
\begin{equation}
  \label{eq:Nk+1Def}
  N_{k+1} :=
  \begin{bmatrix}
    Q_1 \,|\, Q_2 V_2
  \end{bmatrix},
\end{equation}
which finally allows us to draw the following conclusion.

\begin{theorem}
\label{Theo:7}
  $ N_{k+1} $ is an orthogonal basis for $ \cN(C_{k+1})$ such that $
  Q_1 $ is an orthogonal basis for the extended syzygies and $ Q_2V_2
  $ is an orthogonal basis for pure syzygies. 
\end{theorem}

\begin{proof}
  We start with the observation that 
  \[
  C_{k+1}\begin{bmatrix} Q_1&Q_2V_2\end{bmatrix}=0
  \]
  since
  $ V_2 $ is a basis for $ \nN(C_{k+1}Q_2) $ and $ Q_1 $ is a basis
  for $ \rR(A) \subseteq \cN (C_{k+1)} $. This yields that
  \[
  \Span \begin{bmatrix} Q_1 & Q_2 V_2\end{bmatrix} \subseteq
  \nN(C_{k+1}).
  \]
  On the other hand,
  \[
  0 = \rank C_{k+1} \, Q_2 \, V_2 \leq  \rank C_{k+1} + \rank  Q_2
  V_2-q',
  \]
  hence
  \[
  \dim \nN(C_{k+1})= q' - \rank C_{k+1} \leq \rank Q_2 \, W_2 \leq
  \rank \begin{bmatrix} Q_1 & Q_2 V_2\end{bmatrix},
  \]
  yields that $\rank \begin{bmatrix} Q_1 & Q_2 \,
    W_2\end{bmatrix}=\dim \nN(C_{k+1})  $, and
  this completes the proof. 
\end{proof}
\vspace{0.5cm}
\noindent
Theorem~\ref{Theo:7} can be immediately translated into Algorithm~\ref{alg:PuSyz} to extract an orthogonal basis for $\cN(C_{k+1})  $ and consequently pure syzygies. $ \tau$ and $\tau' $ which are used to decide the numerical ranks are introduced in Section~\ref{sec:5}.
\begin{algorithm}[H]
\caption{{\sc Syzygy update}}
\label{alg:PuSyz}
\begin{algorithmic}
{\small
    \REQUIRE { $ N_k $}
    \ENSURE {orthogonal basis $ N_{k+1} $ } 
    \STATE {construct the block diagonal matrices $ L_{j} $ for $ j=1,\dots, n $} 
    \STATE {$ A\leftarrow \begin{bmatrix}
  L_1 &\cdots& L_n
\end{bmatrix}\, N_k $}
    \STATE {$ QSW^{H}\leftarrow $ SVD$ (A),\; r\leftarrow \max \{r : s_{r}>\tau\} $}
    \STATE {$  \begin{bmatrix}
  Q_1 \,|\, Q_2
\end{bmatrix}\leftarrow Q,\; Q_1 \in \CC^{d \times r}, \, Q_2 \in \CC^{d \times
  d-r} $}
        \STATE {$ B\leftarrow C_{k+1}Q_{2} $}
        \STATE {$U\Sigma V^{H}\leftarrow $ SVD$ (B),\; r'\leftarrow \max \{r' : \sigma_{r'}>\tau'\} $}
        \STATE {$  \begin{bmatrix}
  V_1 \, |\, V_2
\end{bmatrix}\leftarrow V,\; V_1 \in \CC^{d' \times r'}, \, V_2 \in \CC^{d' \times d'-r'}  $}
         \STATE {$ N_{k+1}\leftarrow \begin{bmatrix}
    Q_1 \,|\, Q_2 V_2
  \end{bmatrix} $}
}
\end{algorithmic}
\end{algorithm}

\noindent
It is now possible to formulate the degree by degree approach to compute
H-bases using the iterative orthogonalization scheme of Theorem~\ref{Theo:7}. In the next section first we remind a numerically
description of reduction algorithm presented in \cite{Sauer01}, then
we give an updated version of H-bases algorithm using Theorem~\ref{Theo:7}. 

\section{Reduction and H-basis}
To describe the reduction algorithm, we fix, according to
\cite{Sauer01}, an inner product $ (\cdot,\cdot) : \PP\times \PP \rightarrow
\KK $. Keeping the terminology of the previous section, the orthogonal
complement of $  \cM_k(F) $ with respect to $(\cdot,\cdot)$ is defined as 
\[
\cW_k(F):=\PP_{k}\ominus \cM_k(F)=\{f \in \PP_{k} : (f,\cM_k(F))=0\},
\]
or, equivalently,
\[
\PP_{k}=\cM_k(F)\oplus\cW_k(F).
\]
These orthogonal vector spaces enable us to decompose every
homogeneous polynomial $ g\in \PP_k $ in two parts, a part in $
\cM_k(F) $ and another part in $\cW_k(F)$ orthogonal to it, in which
makes a homogeneous remainder $ r^{k} $:
\begin{equation}
  \label{eq:DecHom}
  g=\underbrace{\sum_{j=1}^{p} c_j v_j }_{\in
    \cM_k(F)}+\underbrace{\sum_{j=1}^{q} c_{p+j} w_j }_{\in
    \cW_k(F)}:=\sum_{j=1}^{p} c_j v_j+r^{k}. 
\end{equation} 
This is the main idea of the reduction algorithm which gives an
orthogonal decomposition of every polynomial $ p\in \PP $ as
\begin{equation}
  \label{eq:Dec}
  p=\sum_{f\in F}q_ff+r, \qquad
  \deg (q_ff)\leq \deg(p), \qquad
  r\in \bigoplus_{k=0}^{\deg(p)} \cW_k(F). 
\end{equation}
\noindent
Taken together, Definition~\ref{D:CkDef} and Theorem~\ref{Theo:SVD}
help us to find a basis for $\cM_k(F) $ as well as $ \cW_k(F)$ and
eventually the representation (\ref{eq:DecHom}). Indeed if $C_k(F) =U\Sigma V^{H}  $ is an SVD and $ p=\rank  C_k(F)  $, then 
\[
\cM_k(F)=\Span \{u_1,\dots, u_p\}, \qquad \cW_k(F)=\Span
\{u_{p+1},\dots, u_q\}.
\]
Thus, the coefficient vector $c = \left( c_j : j=1,\dots,p \right)$ in
(\ref{eq:DecHom}) is obtained as a solution 
of the linear system 
\[
\left[u_1, \dots, u_{p}, u_{p+1}, \dots, u_q\right] c = g.
\]
On the other hand, the thin SVD from \cite{GolubvanLoan96} yields
\begin{equation}
 \label{eq:thin-SVD}
 u_j=\dfrac{1}{\sigma_{j}}C_k(F)v_j,\qquad j=1, \dots, p
\end{equation}
and 
replacing (\ref{eq:thin-SVD}) in (\ref{eq:DecHom}) results in
\begin{equation}
  g = \sum_{j=1}^{p} c_j u_j+r^{k}
  = C_k(F) \underbrace{\left( \sum_{j=1}^{p} \dfrac{c_j}{\sigma_j} v_j
    \right)}_{=:\tc}+r^{k}=C_k(F)\tc+r^{k}.
\end{equation}
Partitioning $\tilde{c} = \left( \tc_f : f \in F \right)$ according to the
blocks of $ C_k(F) $ leads to 
\[
g=\sum_{f\in F} C_k(F)\tc_{f}+r^{k},\qquad\;\tc_{f}\in \KK^{\dim
  \PP_{k-\deg(f)}};
\]
by adding proper zeros we finally transition each $ \tc_{f} $ to
\[g_{f,k}:= \begin{bmatrix} 0\\
  \tc_{f} \end{bmatrix}\in \KK^{\dim \PP_{\leq k-\deg(f)}},
\]
which results in the representation 
\[
g=\sum_{f\in F} g_{f,k} \lf(f)+r^{k}.
\]
Repeating this process for each homogeneous part of $ p $ gives us the representation (\ref{eq:Dec}) where
\[q_f=\sum_{k=0}^{\deg(p)}\sum_{f\in F} g_{f,k}, \qquad\qquad r=\sum_{k=0}^{\deg(p)} r^{k}.\]
This leads to the following definition
\begin{definition}
For a given sequence of polynomials $ F $, a polynomial $ p \in \PP $ is called \textit{reducible} module $ F $ and denoted by
$ p\rightarrow _{F} r$
if there exists the representation (\ref{eq:Dec}) for $ p $. In addition, we call $ r $ the \textit{reduced form} of $ p $ w.r.t. $ F $. 
\end{definition}
\begin{remark}
 \label{Re:unique-rem}
It should be mentioned that this definition implies that the
\textit{reduced form} of a polynomial is only dependent on the order
of polynomials and inner product. Moreover it is illustrated in
\cite{Sauer01} that if $ F $ is replaced by an H-basis then the
\textit{reduced form} is unique for a fixed inner product. 
This process can be summarized in the following algorithm
\end{remark}
\begin{algorithm}[H]
\caption{{\sc Reduction }}
\label{alg:Reduction}
\begin{algorithmic}
{\small
    \REQUIRE {polynomial system $ F=\{f_1, \dots, f_s\}\subseteq \PP $,\; $ p\in \PP $ }
    \ENSURE {$ q_f\in \PP,\; r\in \bigoplus_{k=0}^{\deg(p)}\cW_k(F) $ with $ p=\sum_{f\in F} q_ff+r $ }
    \STATE {$r\leftarrow 0, \; q_f\leftarrow 0,\; f\in F$}
    \WHILE{$ p\neq 0 $} 
          \STATE {$ g\leftarrow\lf(p),\; k\leftarrow\deg(p) $}
          \STATE {construct $ C_{k}(F) $}
          \IF {$ C_{k}(F)\neq 0 $}
                 \STATE {decompose $ C_{k}(F)=U\Sigma V^{H} $ using SVD}
                 \STATE {compute homogeneous decomposition $ g=\sum_{f\in F} g_{f} \lf(f)+r^{k} $ }
                 \STATE {$ r\leftarrow r+r^{k} $}
                 \STATE {$ p\leftarrow p-\sum_{f\in F} g_f f - r^{k} $}
                 \STATE {$ q_f\leftarrow q_{f}+g_{f} $}
          \ELSE
                  \STATE {$ r=r+g $}
          \ENDIF
    \ENDWHILE
}
\end{algorithmic}
\end{algorithm}

\noindent
Now the pseudo-code for the update version of H-basis algorithm is shown in Algorithm~\ref{alg:H-basis}. The algorithm starts for the initial degree $ k=\min (d_i : 1\leq i \leq s) $. An orthogonal basis for $ N_k $ in the early step of each iteration process is computed from the SVD of $ C_{k}(F) $. The subsequent steps of the algorithm are then computing a basis for pure syzygies and reducing the corresponding polynomial using Algorithm~\ref{alg:Reduction}. The updating of the $ bound $ is explained in the following. 

\begin{algorithm}[H]
\caption{{\sc Numerical H-basis }}
\label{alg:H-basis}
\begin{algorithmic}
{\small
    \REQUIRE {polynomial system $ F=\{f_1, \dots, f_s\}\subseteq \PP $ of degree $ d_1, \dots, d_s $ }
    \ENSURE {H-basis for the ideal $\langle F \rangle$ }
    \STATE {$k\leftarrow \min (d_i : 1\leq i \leq s),\; bound\leftarrow 2\max (d_i : 1 \leq i\leq s)$}
    \STATE {$N_{j}\leftarrow\emptyset,\; 0\leq j\leq k-1$}
    \WHILE{$ k\ \leq\ bound $} 
          \STATE {construct $ C_{k}(F) $}
          \IF {$ N_{k-1}= \emptyset $}
                
                 \STATE {$ \widetilde{N_k}:=N_{k}\leftarrow $ an orthogonal basis for $ \nN(C_{k}(F)) $}
          \ELSE
                  \STATE {construct $ \begin{bmatrix}
                                   Q_1 \,|\, Q_2 V_2
                                     \end{bmatrix} $ using Algorithm~\ref{alg:PuSyz}}
                  \STATE {$ N_{k}\leftarrow \begin{bmatrix}
                                   Q_1 \,|\, Q_2 V_2
                                     \end{bmatrix} $, $\; \widetilde{N_k} \leftarrow Q_{2}V_{2}  $}
          \ENDIF
          \WHILE {$ \widetilde{N_k}\neq \emptyset $}
                  \STATE { $p\leftarrow \sum_{i=1}^{s} \sum_{|\alpha| = k-d_i} v_{i,\alpha} \,
                                x^\alpha \, f_i, \; v\in \widetilde{N_k}$} 
                  \IF {$ p\rightarrow_{F} f_{s+1}\neq 0 $} 
                        \STATE {$ F\leftarrow F\cup \{f_{s+1}\},\; s\leftarrow s+1$}
                        \STATE {$ k\leftarrow \deg (f_{s})-1$}
                  \ENDIF 
          \ENDWHILE
          \STATE {update $ bound $}
          \STATE {$ k\leftarrow k+1 $}
    \ENDWHILE
}
\end{algorithmic}
\end{algorithm}

In view of the above algorithm, we need to choose an
appropriate degree \textit{bound} to ensure that a generating set
for $ \Syz(\lf(F)) $ has been constructed. Since we generate syzygies
degree by degree, reaching such a bound tells us that all syzygies
reduce to zero and thus the polynomial system is indeed an H-basis.
Finding this \emph{stopping criterion} has already been discussed in
\cite{batselier14:_canon_decom_c_numer_groeb_border_bases,HashemiJavanbaht2017S}
in detail. However, for the sake of completeness, we give a short
analysis here.

A key component in finding the termination degree relies on
Schreyer's theorem. Schreyer showed in his diploma thesis
\cite{schreyer80:_berec_syzyg_weier_divis} that a generating
set, more precisely, even a Gr\"obner basis, for the syzygy module of a
Gr\"obner basis is obtained by reducing every S-polynomial of each
pair of polynomials of underlying Gr\"obner basis to zero,
cf. \cite{schreyer80:_berec_syzyg_weier_divis}. This implies that if $
G $ is a Gr\"obner basis for $ \langle G\rangle $, then 
 $ k=\max\{\deg \tau_{ij}: \tau_{ij}=\lcm (\lm(f_i),\lm(f_j)),\ f_{i},
 f_{j}\in G\} $ is a maximal degree on $ h_{i}g_{i}$, where $ (h_{1},
 \dots, h_{m})\in \Syz(G) $ and $ \lm(f) $ stands for a leading
 monomial of a polynomial $ f $ under a given monomial ordering. 

On the other hand, Buchberger's criterion provides an algorithm to
construct a Gr\"obner basis by computing the remainder of each
S-polynomial and adding the non-zero remainders to the candidate set
\cite{CoxLittleOShea92}. Lazard showed that computing such a remainder
is equivalent to bring a resultant matrix into triangular form,
\cite{lazard83:_gro_gauss}. By this argument, the pivots of a row
reduced echelon form of $ C_k(F) $ correspond to the leading
monomials of the reduced Gr\"obenr basis for $ \lf(F)$ for
sufficiently large $ k $, provided that all columns of the transpose $
C_k(F)$ are reversed.
We refer to such a matrix by $ R_{k}(F) $. This tells us
that if we bring $ C_k(F) $ into a triangular form $ R_k(F) $, then
the maximum degree of least common multiple of each pair of leading
monomials provides an upper bound to find all of syzygies of $ \lf(F)
$. For this $ k $ the reduced Gr\"obner basis of $ \lf(F) $ can be
retrieved from $ R_k(F) $ which is discussed in detail in
\cite{batselier14:_canon_decom_c_numer_groeb_border_bases}.
 
This argument gives an approach to find an appropriate degree bound to
terminate the above algorithm. In doing so, we compare
the pivot elements of $ R_{k}(F) $ with the pivots of $ R_{k-1}(F) $
which have been transitioned to the monomials in $ \TT_{k} $ by
multiplying with the shift matrices from
Section~\ref{Sec:result}. If there exists any new leading monomial, we
update the bound by computing the maximum degree of least common
multiple of each pair of leading monomials. Otherwise, if the
algorithm reaches the updated bound and no new leading
monomial appears up to this degree, it follows that a Gr\"obner basis of
$ \cM_{bound}(F) $ and consequently a basis of $ \Syz(\lf(F)) $
is found. In this case, the algorithm cab be terminated if there is no
non-zero reminder of reduction. Since the bound will become stable after
finitely steps due to the ascending chain condition property of $ \PP
$, cf.\cite{CoxLittleOShea92}, the algorithm will always terminate.

\begin{remark}
It should be mentioned that we use the theory of Gr\"obner basis here
only to find the upper bound for the algorithm by simply computing the
row reduced echelon form of a matrix without computing any
Gr\"obner basis or any S-polynomials. The algorithm itself is
still free of computing Gr\"obner bases.
\end{remark}

\begin{remark}
As mentioned in Remark~\ref{Re:unique-rem}, the reduced form of a
polynomial w.r.t. an H-basis is unique. This helps us to solve
the ideal membership problem easily by finding an H-basis first and
then applying reduction to a given polynomial. Finding the common zeros of a set
of polynomials by means of eigenvectors of a generalization of
Frobenius companion matrices is another problem
which is addressed in \cite{moeller01:_multiv} and discussed using
H-basis technique in \cite{MoellerSauer00}.
\end{remark}


\section{Numerical results}
\label{sec:5}

Any numerical implementation of the above algorithms has to rely on a
\emph{tolerance} $ \tau $ 
to decide the numerical rank and a \emph{threshold} $ \varepsilon $ that
determines whether a float number is numerically zero or not. The
latter one is usually chosen as the \emph{machine accuracy} which
depends on the mantissa length of the underlying floating point
arithmetic and is roughly the value $ 2.22\times 10^{-16} $ for 
double precision floating point numbers, the most frequently used arithmetic on
contemporary processors. 
A standard choice for $\tau $ in \cite{GolubvanLoan96} is then given by
$\tau=\max\{n,m\} \max_{j}\sigma_{j} \epsilon$, where $ \epsilon $
the $ \sigma_{j}$ are the singular values of the underlying matrix
with $ m $ columns and $ n $ rows. This defines the meaning of the
term \emph{rank} in Numerical Linear Algebra, often referred to as
\emph{numerical rank}, in contrast to its meaning in Linear
Algebra. In other words, if $ A $ is an $ m\times n $ matrix with
singular values $ \sigma_1\geq\cdots \geq \sigma_n $, then the
numerical rank $ r $ is chosen such that  
\begin{equation}
 \label{def:numer-rank}
 \sigma_1\geq\cdots\geq \sigma_r\geq \tau\geq \sigma_{r+1}\geq \cdots
 \geq \sigma_n 
\end{equation}
cf. \cite{batselier14:_macaul}. For an exact definition and detailed discussion of numerical rank, we
refer to \cite{li09}. (\ref{def:numer-rank}) shows that
the correct determination of the numerical rank strongly depends on a
good choice of $ \tau $. On the other hand, the determination of
numerical rank is a crucial step in Algorithm \ref{alg:PuSyz} to
distinguish new syzygies from extended ones and consequently finding a
correct H-basis. So, the good choice of $ \tau $ to guarantee the
correct result is a must. 

Our numerical tests illustrate that a standard choice of $ \tau $
works truly for most experiments though it fails for some polynomial
systems like Caprasse4. We will observe that in this case there is a
large gap between singular values in index $ k $ but the default value
of $ \tau $ is smaller than $ \sigma_{k} $, so the numerical rank
detection fails and the syzygies are \emph{ill-conditioned}.
The ratio $ \sigma_{r}/\sigma_{r+1} $, the so called
approx-rank gap influences the accuracy of rank-revealing computation.
This is discussed in details in
\cite{demmel15:_commun_qr,foster06:_compar}. Indeed a well-defined
numerical rank leads to a choice of $ r $ which maximises this approx-rank
gap. 
 
Apart from rank revelation, also the choice of the threshold $
\varepsilon $ clearly affects the outcome of the algorithm. An
inappropriate value for $ \varepsilon $ results in apparently very
small remainders in the reduction, which forces the algorithm to end
up with the basis $1$, falsely claiming that the polynomials have no
common zeros. To avoid this problem, we have to increase the value
$\varepsilon$ to the suitable value which then yields the correct
result. We will discuss this effect in details for a polynomial system
named after R.~Sendra. 

\subsection{Experiment setup}
Here we compare the efficiency and the stability of the numerical H-basis
algorithm with the symbolic algorithm presented in
\cite{HashemiJavanbaht2017S}.
All numerical experiments are carried out on a 2.5 GHz seven-core
personal computer with 8 GB RAM using 64-bit Matlab and usually the
machine precision $ \epsilon\approx 2.22\times 10^{-16} $. Our
numerical H-basis algorithm is implemented as a Matlab module {\sc
  H-Basis} that is electronically available from the authors upon
request. 

In the first group of experiments we will show how much a numerical
implementation can speed up the degree-by-degree approach to compute
H-bases. The capability of the algorithm is evaluated for a benchmark
set of examples with different Krull-dimensions which confirms that
our approach is not restricted to zero dimensional ideals. In the
second part, however, we will discuss the problems arising from the
floating numbers in the numerical implementation. In each experiment
we use the Hilbert polynomial of the ideal to check the correctness of
the obtained numerical H-basis. Indeed the Hilbert polynomial of an
ideal equals to the Hilbert polynomial of the leading form ideal,
cf. \cite{KreuzerRobbiano00}.

In what follows, the run time is measured in seconds. The H-basis
column shows that how many polynomials are detected in corresponding
H-basis. $ d_{max} $ is the maximum degree of polynomials in the
H-basis and $ bound $ shows the degree that algorithm is
terminated. The names of the benchmark ideals are due to
\begin{quote}
  {\tt
  http://homepages.math.uic.edu/$ \sim $jan/Demo/}. 
\end{quote}

\subsection{Correct experiments}
We begin by listing some examples where the numerical method performed
correctly and obtained the proper H-basis, at least when the
thresholding parameter was chosen properly.

\par\medskip
\par\noindent\textbf{Weispfenning94.}
In the first numerical experiment, the capability of the algorithm is
tested for a 0-dimensional ideal which consists of 2 polynomials of
total degree 5 and a polynomial of degree 4 in 3 variables. The
numerical approach speeds up the calculation of the H-bases by a
factor of 96.
\[\left\lbrace 
\begin{array}{cr}
f_1: & xy^2z +y^4 +x^2  -2xy  +y^2 +z^2 \\
f_2: &  -x^3y^2 +xyz^3 +xy^2z +y^4  -2xy \\
f_3: &  xy^4 +yz^4  -2x^2y  -3
\end{array}
\right.
\]

\begin{center}
\begin{tablehere}
\begin{tabular}{ |c||c|c|c|c|c| } 
\hline
Weispfenning94 & $ \varepsilon $ & time &  H-basis & $d_{max}$ & $ bound $\\
\hline
{\sc Symbolic} &- & 2507 & 6 &  6& 18 \\
\hline
 & $ 10^{-12} $& 26.3 & 11 &  &    \\
 
  {\sc Numerical} & $ 10^{-10} $& 26.3 & 11 & 6 & 18   \\

                     & $ 10^{-8} $& 26.7  & 11 &  &    \\
\hline
\end{tabular} 
\end{tablehere} 
\end{center}

\par\noindent\textbf{Liu.}
For the second numerical experiment, we consider a one dimensional
polynomial system consists of 4 polynomials in 5 variables of total
degree 2.
\[\left\lbrace 
\begin{array}{crcc}
f_1: & yz-zw-x+u  \\
f_2: & -xy+yw-z+u \\
f_3: &  xw-zw-y+u  \\
f_4: &  -xy+xz-w+u 
\end{array}
\right.
\]

\begin{center}
\begin{tablehere}
\begin{tabular}{ |c||c|c|c|c|c| } 
\hline
Liu &$ \varepsilon $ & time  &  H-basis & $ d_{max} $ & $ bound $ \\
\hline
{\sc Symbolic} &- & 526  & 5 & 2 & 8 \\
\hline
        & $ 10^{-12} $& 13.44 & 5 &  &    \\

 {\sc Numerical} & $ 10^{-10} $& 13.6   & 5 & 2 & 8   \\

            & $ 10^{-8} $& 13.26 & 5 &  &    \\
\hline
\end{tabular}
\end{tablehere}
\end{center}

\par\noindent\textbf{Gerdt2.}
In the third numerical experiment, we consider a three dimensional
polynomial system that consists of 2 polynomials in 5 variables of
degree 4.
\[\left\lbrace 
\begin{array}{crcc}
f_1: &\hspace{-0.2cm} 5xy^3-140y^3z-3x^2y+45xyz+210y^2w-420yz^2-25xw+126yu+70zw  \\
f_2: &\hspace{-0.2cm} 35y^4-30xy^2-210y^2z+3x^2+30xz+140yw-105z^2-21u 
\end{array}
\right.
\]

\begin{center}
\begin{tablehere}
\begin{tabular}{ |c||c|c|c|c|c|c| }
\hline
Gerdt2 & $ \varepsilon $ & time  &  H-basis  & $ d_{max} $ & $ bound $ \\
 \hline {\sc Symbolic} &- & 3840   & 6 & 6 & 12\\
\hline
 &$10^{-10}$ & - & constant &  &  \\
 {\sc Numerical}  &$10^{-9}$ & 165 & 6 & 6 & 12 \\
 &$10^{-6}$ & 165 & 6 &  &  \\
\hline
\end{tabular}
\end{tablehere}
\end{center}
The following tables show further results of successful runs. The
polynomial system Schwartz can be found in
\cite{decker99:_primar_decom} and Lazard-Mora is defined in
\cite{lazard83:_gro_gauss}. All the other polynomial systems in this
section and the next section are available at the aforementioned
website.

\begin{center}
\begin{tablehere}
\begin{tabular}{ |c||c|c|c| } 

\hline
Lorentz & $ \epsilon $ & time &    H-basis \\
\hline
{\sc Symbolic} &- & 6s & 5   \\
\hline
 {\sc Numerical} & $ 10^{-10} $ & 3.05 & 5  \\
\hline

\multicolumn{1}{c}{}  \\
\hline
Conform1 & $ \epsilon $ & time &  H-basis  \\
\hline
{\sc Symbolic}& - & 19  & 10   \\
\hline
{\sc Numerical} & $ 10^{-10} $ & 1.82 & 10  \\
\hline

\multicolumn{1}{c}{}  \\
\hline
Redeco5 & $ \epsilon $ & time &  H-basis \\
\hline
{\sc Symbolic}  & - & 252 & 5   \\
\hline 
 {\sc Numerical} & $ 10^{-10} $ & 34.7  & 5  \\
\hline

\multicolumn{1}{c}{}  \\
\hline
Noon & $ \epsilon $ & time  &  H-basis \\
\hline
{\sc Symbolic} & -& 7660  & 7   \\
\hline
 {\sc Numerical} &$ 10^{-10} $ & 172.1 & 9 \\
\hline

\multicolumn{1}{c}{}  \\
\hline
Schwartz & $ \epsilon $ & time &  H-basis \\
\hline
{\sc Symbolic} & -& 18672 & 6   \\
\hline
 {\sc Numerical} &$ 10^{-10} $ & 183.3 & 6 \\
\hline

\multicolumn{1}{c}{}  \\
\hline
Lazard-Mora & $ \epsilon $ & time &  H-basis \\
\hline
{\sc Symbolic} & -& $ > 2 $ hours & 3   \\
\hline
 {\sc Numerical} &$ 10^{-10} $ & 338.2  & 3 \\
\hline

\end{tabular}
\end{tablehere}
\end{center}

\subsection{Ill-conditioned ideals}
Next, we show some examples where numerical ill-conditioning occurred
and discuss the reasons for failure in some more detail.

\par\medskip
\par\noindent\textbf{Sendra.}
In this example we illustrate the failure of the numerical algorithm to find an H-basis because of small remainders. The polynomial system consists of 2 polynomials in two variables of total degree 7 with Krull-dimension 0. 
\begin{center}
\begin{tablehere}
\begin{tabular}{ |c||c|c|c|c|c| } 
\hline
Sendra & $ \varepsilon $ & time &  H-basis & $ d_{max} $ & $ bound $   \\
\hline
{\sc Symbolic} &- & 6  & 4 & 12 & 22 \\
\hline 
  {\sc Numerical}       &$ 10^{-10} $ & -  & constant  & - & -\\
     
       &$ 10^{-4} $ & 1.15  & 4  & 12 & 22 \\
\hline
\end{tabular}
\end{tablehere}
\end{center}
The algorithm starts with $ k=7 $ but first syzygies appear in $ k=13 $. $ C_{13}(F) $ is a $ 14\times 14 $ matrix with density $ 28\% $ and singular values 
\[   \sigma_{13}= 0.364> \tau\approx 10^{-11} > 1.7\times 10^{-11} = \sigma_{14}. \]
Thus, $ \rank C_{13}(F)=13 $ and $ \dim \cN(C_{13}(F))=1 $. The
corresponding syzygy is reduced to the non-zero remainder of degree 11
and 2-norm $ 4.039\times 10^{3} $. The second non-zero remainder is
detected in $ k=12 $, where $ C_{12}(F) $ is a $ 13\times 14 $ matrix
with density $ 38\% $, the tolerance and singular values 
\[ \sigma_{12} = 0.029> \tau\approx 10^{-11}> 1.14\times 10^{-13} = \sigma_{13}. \]
The numerical rank is determined by the default tolerance as
12. Therefore, the null space satisfies $\dim \cN(C_{12}(F))=2 $ and
the corresponding
polynomial to this new syzygy gives a non-zero remainder of total
degree 11 and 2-norm $ 9.37\times 10^4 $. The third step of the
algorithm is started with $ k=11 $, and the first new syzygies are
appeared in $ k=12 $. $ C_{12}(F) $ is a $ 13\times 16 $ matrix with $
\dim \cN(C_{12}(F))=3 $. The exact computations in Maple show that
their corresponding polynomials are reduced to zero. However in Matlab
we have a non-zero remainder of degree 10 and 2-norm $ 1.744 \times
10^{-6} $. Continuing this process in the next steps results a very
small non-zero remainder in each step but still greater than threshold
$ \varepsilon $, so that the algorithm ends up with a constant
non-zero remainder $ -1.36\times 10^{-10} $. To get ride of this
obstacle we increased the $ \varepsilon $ from $ 10^{-10} $ up to $
10^{-4} $ and observed that the algorithm terminates in two steps with
a correct H-basis. 
\begin{remark}
It is worth mentioning that if we divide each polynomial in the initial polynomial system as well as the obtained remainders in each step by 1-norm, 2-norm and $ \infty $-norm, then all of small remainders will be vanished and we will obtain a correct H-basis in two steps for $ \varepsilon=10^{-9} $. It means that in case that the algorithm fails to get a correct H-basis due to the small remainders, normalizing the polynomials can counteract the bad affect of small reminders by minifying the coefficients of underlying polynomials in the reduction process and give the correct result.  
\end{remark}
 
\par\noindent\textbf{Caprasse4.}
The polynomial system so-called Caprasse4 demonstrates the failure of rank-revealing for a default numerical tolerance. Our observations show that the failure of rank-revealing during one of the iterations destroys the result of all of the consequent iterations in a way that the algorithm ends up with very small non-zero constant which is not an acceptable H-basis. This polynomial system consists of 4 polynomials of degrees 3,3,4,4 on 4 variables with Krull-dimension 0.

For the first observation we suppose that $ \varepsilon =10^{-10} $ and observe that the first failure of rank-revealing occurs in step = 4 (each step is whenever a new non-zero reminder is added to the candidate set) at $ k=5 $. Inspecting $ \sigma_{37}=0.0017 $ and $ \sigma_{38}=1.56\times 10^{-12} $ shows that the numerical rank should be $ 37 $ instead of $ 40 $ although the default tolerance is $ 8.53\times 10^{-14} $. It means that $ \dim \cN(C_{5}(F))=0 $ and hence $ \dim\Syz_{5}(\lf (F))=0 $. The symbolic implementation 
however shows that the rank is $ 37 $ and $ \dim \cN(C_{5}(F))=3 $. In spite of this numerical rank, the polynomials corresponding to these three detected syzygy in Maple algorithm is reduced to zero. 
At $ k=6 $ the numerical rank is estimated to be $ 79 $. While the approx-rank gap is maximised at 74 with $ \sigma_{74}/\sigma_{75}=1.71\times 10^{8} $.
 It implies $ \dim \Syz_{6}(\lf (F))=11 $ with a non-zero remainder of 2-norm $ 3.05\times 10^{-10} $.
 If we let the execution of the algorithm continues, very small non-zero remainders detected in the next steps (of 2-norm almost $ 3.8\times 10^{-9} $) make the algorithm to end up with a non-zero but very small constant. 
 
 To release the affect of small non-zero remainders and exploring the impact of wrong rank-revealing we increased the threshold to $ \varepsilon=10^{-7} $. As we expect we will not have any non-zero remainder in step = 4 at $ k=6 $ and algorithm proceeds with $ k=7 $. Having syzygies at $ k=6 $, a $ 40\times 44 $ matrix $ A $ of extended syzygies is generated. It's singular values show that the rank-revealing determines the numerical rank correctly to be $ 44 $. The approx-rank gap $ \sigma_{110}/\sigma_{111}=2.85\times 10^{9} $ of matrix $ B $ however shows that another wrong rank-revealing has occurred since the numerical rank is estimated to be 112. Regardless of fault of rank-revealing both symbolic and numerical algorithms find a non-zero remainder of total degree 5. The non-zero remainder computed in Maple has 2-norm $ 17.06 $ while 2-norm of non-zero remainder of Matlab implementation is 0.34. So, step 5 starts from $ k=5 $ with the same values for singular values and tolerance as reported for $ \varepsilon =10^{-10} $. For $ k=6 $ approx-rank gap at 78 is $ \sigma_{78}/\sigma_{79}=2.96\times 10^{8} $ which illustrates the numerical rank should be 78 but it is estimated to be 79 instead. Here all of new detected syzygies are reduced to zero. In symbolic implementation though a non-zero remainder of degree 4 is detected.
 Tracing the algorithm in the next steps show that the next non-zero remainder is of total degree 5 which is detected at $ k=7 $. In the next steps of the algorithm the small remainders are appeared such that again the algorithm is ended up with a non-zero small remainder. 
 
 The above observations show that even a minor error in the
 rank-revealing in one step is caused the magnitude faults in the
 following steps and eventually the wrong result. This comes back to
 the sensitivity of the floating numbers as well as the
 rank-revealing. In the case that there is a large gap in singular
 values at index $ k $ but the tolerance is larger than $ \sigma_{k} $
 the other rank-revealing methods such as L-curve analysis are the
 better choice to determine the numerical rank
 cf.~\cite{hansen98:_rank}.
 
 
\section{Conclusion}
We have presented a numerical stable algorithm to compute H-bases. The
approach benefits the SVD to compute a minimal generating set for the
syzygy of a given degree $k$ of the underlying ideal which speeds up
computing H-bases dramatically.
This considerable
achievement in run time comes at the cost of losing an accurate
H-basis in some examples because of the rank-revealing fault or very
small remainders due to the floating numbers property. But these kind
of ideals do not appear frequently in practice. 

It is worthwhile to be mentioned that as the proposed approach
computes a minimal generating set for syzygies in each degree, it can
be applied to  compute the second syzygies, third syzygies and etc. It
means that for a given ideal we can make the finite chain of syzygies
(with the length at most equals to the number of variables), so-called
free resolutions which are very important tools in commutative and
computational algebra and releases important invariants such as Betti
numbers, Hilbert regularity and Krull-dimension,
cf.~\cite{peeva11:_graded_syzyg}. To deal with this issue needs very
precise discussion which is left to the future for further discussion.

\end{document}